\documentclass[11pt,english]{amsart}

\usepackage{amsmath,leftidx,amsthm}
\usepackage[all]{xy}
\usepackage{xspace}
\usepackage[psamsfonts]{amssymb}
\usepackage[latin1]{inputenc}
\usepackage{graphicx,color}
\usepackage{hyperref,fancyhdr}
\usepackage{fullpage}

\definecolor{new}{rgb}{0.7,0,0.6}

\newtheorem*{question}{\bf Question}

\newtheorem{theorem}{\bf Theorem}[section]
\newtheorem{proposition}[theorem]{\bf Proposition}
\newtheorem{definition}[theorem]{\bf Definition}
\newtheorem{Theorem}{\bf Theorem}

\newtheorem{lemma}[theorem]{\bf Lemma}
\newtheorem{corollary}[theorem]{\bf Corollary}

\newtheorem{Definition}[Theorem]{\bf Definition}

\theoremstyle{remark}
\newtheorem*{remark}{\bf Remark}
\newtheorem{exmp}{\bf Example}

\def\C{{\mathbb C}}

\def\R{{\mathbb R}}
\def\Z{{\mathbb Z}}

\def\Q{{\mathbb Q}}

\def\D{\mathbb{D}}

\def\p{\mathbb{P}}

\def\rat{\textup{Rat}}

\def\PGL{\textup{PGL}}

\def\Hom{\textup{Hom}}

\newcommand{\col}{\,{:}\,}

\usepackage{versions}
\usepackage{textcomp,listings}

\definecolor{green}{rgb}{0,0.5,0}
\definecolor{dkgreen}{rgb}{0,0.6,0}
\definecolor{gray}{rgb}{0.5,0.5,0.5}
\definecolor{mauve}{rgb}{0.58,0,0.82}
\lstnewenvironment{python}[1][]{
\lstset{
  frame=single,                   % adds a frame around the code
  language=python,                          % the language of the code
  basicstyle=\ttfamily\small, % the size of the fonts that are used for the code
  numbers=left,                             % where to put the line-numbers
  numberstyle=\scriptsize\color{black},  % the style that is used for the line-numbers
  stepnumber=1,                   % the step between two line-numbers. If it's 1, each line
  numbersep=7pt,                  % how far the line-numbers are from the code
  backgroundcolor=\color{white},      % choose the background color.
  stringstyle=\color{red},
  showspaces=false,               % show spaces adding particular underscores
  showstringspaces=false,         % underline spaces within strings
  showtabs=false,                 % show tabs within strings adding particular underscores
  tabsize=2,                      % sets default tabsize to 2 spaces
  captionpos=b,                   % sets the caption-position to bottom
  breaklines=true,                % sets automatic line breaking
  breakatwhitespace=false,        % sets if automatic breaks should only happen at whitespace
  rulecolor=\color{black},        % if Ä set, the frame-color may be changed on line-breaks within not-black text (e.g. commens (green here))
  framexleftmargin=1mm, framextopmargin=1mm, frame=shadowbox, %rulesepcolor=\color{blue},#1
  alsoletter={1234567890},
  otherkeywords={\ , \}, \{},
  keywordstyle=\color{blue},       % keyword style
  stringstyle=\color{mauve},         % string literal style
  emph={access,and,break,class,continue,def,del,elif ,else,%
  except,exec,finally,for,from,global,if,import,in,i s,%
  lambda,not,or,pass,print,raise,return,try,while},
  emphstyle=\color{black}\bfseries,
  emph={[2]True, False, None, self},
  emphstyle=[2]\color{blue},
  emph={[3]from, import, as},
  emphstyle=[3]\color{blue},
  upquote=true,
  morecomment=[s]{"""}{"""},
  commentstyle=\color{dkgreen},       % comment style
  emph={[4]1, 2, 3, 4, 5, 6, 7, 8, 9, 0},
  emphstyle=[4]\color{blue},
  literate=*{:}{{\textcolor{blue}:}}{1}%
  {=}{{\textcolor{blue}=}}{1}%
  {-}{{\textcolor{blue}-}}{1}%
  {+}{{\textcolor{blue}+}}{1}%
  {*}{{\textcolor{blue}*}}{1}%
  {!}{{\textcolor{blue}!}}{1}%
  {(}{{\textcolor{blue}(}}{1}%
  {)}{{\textcolor{blue})}}{1}%
  {[}{{\textcolor{blue}[}}{1}%
  {]}{{\textcolor{blue}]}}{1}%
  {<}{{\textcolor{blue}<}}{1}%
  {>}{{\textcolor{blue}>}}{1},%
}}{}

\excludeversion{code}

\def\cal{\mathcal}
\newcommand\mycom[2]{\genfrac{}{}{0pt}{}{#1}{#2}}

\def\and{{\quad\text{and}\quad}}

\title{Symmetrization of rational maps: arithmetic properties and families of Latt\`es maps of $\p^k$}
\author{Thomas Gauthier}
\address{LAMFA, Universit\'e de Picardie Jules Verne, 33 rue Saint-Leu, 80039 AMIENS Cedex 1, FRANCE}
\email{thomas.gauthier@u-picardie.fr}
\author{Benjamin Hutz}
\address{Saint Louis University, 220 N. Grand Blvd., St.Louis, MO 63103, USA}
\email{benjamin.hutz@slu.edu}
\author{Scott Kaschner}
\address{Butler University, 4600 Sunset Ave., Indianapolis, IN 46208, USA}
\email{skaschne@butler.edu}

\thanks{The first author is partially supported by the ANR grant Lambda ANR-13-BS01-0002. The second author was partially supported by NSF grant DMS-1415294.}

\subjclass[2010]{
37F10, %Polynomials; rational maps; entire and meromorphic functions
37P35, %Arithmetic properties of periodic points
37F99, %None of the above, but in this section
37F45, %Holomorphic families of dynamical systems; the Mandelbrot set;bifurcations
37P45 %Families and moduli spaces
}

\keywords{dynamical system, post-critically finite, preperiodic point, symmetric product, lattes map}

\begin{document}

\begin{abstract}
In this paper we study properties of endomorphisms of $\p^k$ using a symmetric product construction $(\p^1)^k/\mathfrak{S}_k \cong \p^k$. Symmetric products have been used to produce examples of endomorphisms of $\p^k$ with certain characteristics, $k\geq2$. In the present note, we discuss the use of these maps to enlighten
%arithmetic phenomena and
stability phenomena in parameter spaces. In particular, we study %notions of uniform boundedness of rational preperiodic points via good reduction information,
$k$-deep post-critically finite maps and
characterize families of Latt\`es maps.
\end{abstract}

\maketitle

\tableofcontents

\section*{Introduction}

In this paper we study dynamical properties of endomorphisms of $\p^k$, mainly focusing on the dynamics of a symmetric product construction $(\p^1)^k/\mathfrak{S}_k \cong \p^k$. Symmetric products have been used to produce examples of endomorphisms of $\p^k$ with certain characteristics when $k\geq2$ and their dynamics over the field $\mathbb{C}$ of complex numbers is rather simple and completely understood (see \cite{dinhsibony2,fornaesssibony,ueda}, for example). In the present note, we discuss the use of these maps to enlighten stability phenomena in parameter spaces. In particular, we study $k$-deep post-critically finite maps and
characterize families of Latt\`es maps (in particular, those containing symmetric products).

Let $d\geq2$ be an integer. We denote by $\Hom_d(\p^k)$ the space of endomorphisms of $\p^k$ of degree $d$. Such  endomorphisms may be represented in homogeneous coordinates by $k+1$ degree $d$ homogenous polynomials.  This shared degree $d$ is called the algebraic degree of the endomorphism.  Unless otherwise specified, ``degree'' will mean ``algebraic degree'' throughout this paper. The space $\Hom_d(\p^k)$ is known to be a smooth irreducible quasi-projective variety of dimension $N_d(k):=(k+1)\frac{(d+k)!}{d!\cdot k!}-1$. More precisely, it can be identified with an irreducible Zariski open set of $\p^{N_d(k)}$, see e.g. \cite[\S 1.1]{BB1}.
~

Here, we mainly investigate properties of the following construction.
\begin{Definition}
For any $f\in\Hom_d(\p^1)$, we define the $k$-\emph{symmetric product} of $f$ as the endomorphism
$F:\p^{k}\longrightarrow\p^{k}$ making the following diagram commute
\begin{equation}\label{eq_comm_sq}\xymatrix {\relax
(\p^1)^k \ar[r]^{(f,\ldots,f)} \ar[d]_{\eta_k} & (\p^1)^k\ar[d]^{\eta_k} \\
\p^{k} \ar[r]%@{-->}[r]
    _{F} & \p^{k}}
\end{equation}
\end{Definition}
For the map $\eta_k$ in the diagram, we use the regular map defined as the quotient map of the action of the symmetric group $\mathfrak{S}_k$ on $(\p^1)^k$ by permuting the terms of the product.

In Section \ref{sect_background} we prove various basic properties of the dynamics of the map $F$ in relation to $f$, including that $F$ is an endomorphism (as suggested in the definition). Most of the properties we give seem classical and should follow from general properties of quotients under the symmetric group, such as from Maakestad \cite{maakestad}.
In the remaining sections, we apply this construction to explore properties of post-critically finite symmetric product endomorphisms of $\p^k$ and deformations of Latt\`es maps of $\p^k$.

\medskip

\medskip

\par\noindent\textit{Post-critically finite symmetric products.}
\par In Section \ref{sect_pcf} we consider another problem. Recall that $F\in\Hom_d(\p^k)$ is \emph{post-critically finite} if the \emph{post-critical set} of $F$
$$\mathcal{P}(F):=\bigcup_{n\geq1}F^n(\mathcal{C}(F))~,$$
i.e., the forward orbit of $\mathcal{C}(F)$, the critical locus of $F$, is a strict algebraic subvariety of $\p^k$.

If $X$ and $Y$ are two projective varieties of the same dimension and $g:X\rightarrow Y$ is a surjective morphism with finite fibers, we define the \emph{critical locus} $\mathcal{C}(g)$ of $g$ as the Zariski closure of
\[X_{\mathrm{sing}}\cup\{x\in X_{\mathrm{reg}}\, ; \ d_xg \ \text{ is not invertible}\}.\]

\begin{Definition}
We say that $F\in\Hom_d(\p^k)$ is $1$-\emph{deep post-critically finite} if $F$ is post-critically finite and $F_{(1)}:=F|_{\mathcal{C}(F)}$ is post-critically finite, meaning the orbit under iteration of $F$ of the critical locus $\mathcal{C}({F_{(1)}})$ of $F_{(1)}$ is an algebraic subvariety of pure codimension $2$ of $\p^k$.

We define $j$-deep post-critically finite inductively for $1<j\leq k-1$ by saying that $F$ is $j$-\emph{deep post-critically finite} if it is $(j-1)$-deep post-critically finite and $F_{(j)}:=F|_{\mathcal{C}({F_{(j-1)}})}$ is post-critically finite, i.e. the orbit under iteration of $\mathcal{C}(F_{(j)})$ is an algebraic subvariety of pure codimension $j+1$ of $\p^k$.

We say that $F$ is \emph{strongly post-critically finite} if it is $(k-1)$-deep post-critically finite.
\end{Definition}

Note that for $0\leq j\leq k-1$ the varieties $\mathcal{C}(F_{(j)})$ and $F^\ell(\mathcal{C}(F_{(j)}))$ all have pure codimension $j+1$ by definition.

The notion of $j$-deep post-critically finite maps has been studied in slightly different contexts.  In \cite{Jonsson-finite}, Jonsson uses the notation $\mathcal E_1=\cap_{i>0}f^i(\mathcal P(f))$ (which is the limit set of $\mathcal C(f)$ when $\mathcal P(f)$ is closed) and defines a map $f\in\Hom_k(\mathbb P^k)$ to be $1$-critically finite if $\mathcal C(f)$, $\mathcal P(f)$, and $\mathcal E_1$ are all codimension $1$ subvarieties, and if $\mathcal C(f)$ and $\mathcal E_1$ contain no common irreducible component.  Then $j$-critically finite maps are defined inductively.
Indeed, the difference in this definition is that Jonsson requires that $\mathcal C(f)\cap\mathcal E_1$ contains no codimension $1$ irreducible component.  This requirement ensures critical points are \emph{non-recurrent}, hence non-periodic. Note that our definition does not exclude that possibility.  Jonsson's work built on Ueda's study of $2$-critically finite maps in \cite{ueda,ueda2} and showed that post-critically finite endomorphisms of $\mathbb P^2$ are always $2$-critically finite.

Astorg \cite{astorg} defines \emph{post-critically finite all the way down} for endomorphism $f\colon X\rightarrow X$, where $X$ is an analytic set, analogously to our definition of strongly post-critically finite.

We focus on the following problem.

\begin{question}
Is a post-critically finite endomorphism of $\p^k$ necessarily strongly post-critically finite?
\end{question}

In the case when $F$ is a symmetric product, we give a positive answer to this question. This answers a question posed during the 2014 American Institute of Mathematics workshop \emph{post-critically finite maps in complex and arithmetic dynamics} in the specific case of symmetric products. Note that Astorg's work \cite{astorg} gives a positive answer to that question for endomorphisms satisfying a mild regularity condition on the post-critical set.

We prove the following theorem.

\begin{Theorem}
Let $f\in\Hom_d(\p^1)$ and $F$ be the $k$-fold symmetric product of $f$ for $k\geq2$. Then $F$ is strongly post-critically finite if and only if $F$ is post-critically finite if and only if $f$ is post-critically finite.\label{tm:strongpcf}
\end{Theorem}

Given a rational map $F$ with indeterminacy set $I$ such that $\mathcal P(F)$ is disjoint from $I$, the notion of $j$-deep post-critically finite is well-defined.  We also provide such a one-parameter family of degree $4$ rational mappings $F_a:\p^2\dashrightarrow\p^2$ that are post-critically finite but \emph{not} strongly post-critically finite.

\medskip

\medskip

\par\noindent\textit{Families of Latt\`es maps of $\p^k$.}
\par Denote by $\mathcal{M}_d(\p^k)$ the \emph{moduli space} of endomorphisms of $\p^k$ with degree $d$, i.e. the set of $\PGL(k+1)$ conjugacy classes of the space of degree $d$ endomorphisms $\Hom_d(\p^k)$. The space $\mathcal{M}_d(\p^k)$ has been proven to be a geometric quotient (in the sense of geometric invariant theory) by Levy \cite{levyPn}.

In Section \ref{sect_lattes} we examine families of Latt\`es maps in the moduli space $\mathcal{M}_d(\p^k)$ of degree $d$ endomorphisms of $\p^k$. Our aim here is to have a complete description of non-isotrivial families of Latt\`es maps of $\p^k$ when $k\geq2$: an endomorphism $f:\p^k\longrightarrow\p^k$ is a \emph{Latt\`es map} if there exists an abelian variety $\mathcal{T}$ of dimension $k$, an isogeny $D:\mathcal{T}\longrightarrow\mathcal{T}$ and a finite branched cover $\Theta:\mathcal{T}\longrightarrow\p^k$ such that the following diagram commutes
\begin{center}
$\xymatrix {\relax
\mathcal{T} \ar[r]^{D} \ar[d]_{\Theta} & \mathcal{T}\ar[d]^{\Theta} \\
\p^{k} \ar[r]_{f} & \p^{k}}$
\end{center}
Over $\C$, the map $D$ is, in fact, induced by an affine map $\mathrm{D}:\C^k\rightarrow\C^k$ whose linear part is of the form $\sqrt{d}U$ where $U$ is unitary, see~\cite{BertelootLoeb}. By an abuse of notation, we call $\sqrt{d}U$ the \emph{linear part} of $D$.

The complex dynamics of Latt\`es maps has been deeply studied in several beautiful papers (see e.g.~\cite{Milnor,zdunik,BL,BertelootLoeb,bertelootdupont,Dupont-spherical}). The families and the perturbations of Latt\`es maps of $\p^1$ are also quite completely understood (see \cite{Milnor,BB1,Article2}).

~

Recall also that a \emph{family} of degree $d$ endomorphisms of $\p^k$ is a morphism $f:\Lambda\times \p^k\longrightarrow\p^k$, where $\Lambda$ is a quasiprojective variety of dimension $m\geq1$ and for all $\lambda\in\Lambda$, the map $f_\lambda:=f(\lambda,\cdot):\p^k\longrightarrow\p^k$ is an endomorphism of degree $d$. Equivalently, the map $\lambda\in\Lambda \longmapsto f_\lambda\in\Hom_d(\p^k)$ is a morphism of quasiprojective varieties. We also say that $(f_\lambda)_{\lambda\in\Lambda}$ is \emph{isotrivial} if there exists a holomorphic map $\lambda\in\Lambda\longmapsto m_\lambda\in \PGL(k+1,\C)$ and $\lambda_0\in\Lambda$ such that
\[f_\lambda=m_\lambda^{-1}\circ f_{\lambda_0}\circ m_\lambda~ , \ \lambda\in\Lambda~.\]
Denote by $\Pi:\Hom_d(\p^k)\to\mathcal{M}_d(\p^k)$ the canonical projection.

\begin{Definition}
When $(f_\lambda)_{\lambda\in\Lambda}$ is a family of endomorphisms of $\p^k$, we say that it has \emph{dimension $q$ in moduli} if the set $\{\Pi(f_\lambda)\, ; \ \lambda\in\Lambda\}$ is an analytic set of $\mathcal{M}_d(\p^k)$ of dimension $q$. We then set $\dim_{\mathcal{M}}(f_\lambda,\Lambda):=q$.
\end{Definition}

We focus on \emph{maximal} families of Latt\`es maps containing a specific Latt\`es map $f$.

\begin{Definition}
A family $(f_\lambda)_{\lambda\in\Lambda}$ of Latt\`es maps of $\p^k$ containing $f\in\cal{M}_d(\p^k)$ is \emph{maximal} if for any family $(f_t)_{t\in X}$ of Latt\`es maps containing also $f$, we have $\dim_{\mathcal{M}}(f_\lambda,\Lambda)\geq\dim_{\mathcal{M}}(f_t,X)$.
\end{Definition}

\begin{Definition}
A map $f\in\cal{M}_d(\p^k)$ is \emph{rigid} if any maximal family $(f_{\lambda})$ containing $f$ is isotrivial.
\end{Definition}

When $k=1$, the families of Latt\`es maps have been completely classified by Milnor~\cite{Milnor}. He proves that a family $(f_t)$ of Latt\`es maps of $\p^1$ has positive dimension in moduli, i.e. is non-isotrivial, if and only if $d=a^2$ is the square of an integer $a\ge2$ and the affine map inducing any map $f_{t_0}$ is of the form $z\mapsto az+b$.
We aim here at giving a generalization of that precise statement.

Our main result on this problem is the following.
 \begin{Theorem}\label{tm:rigidLattes}
Let $f\in\mathcal{M}_d(\p^k)$ be a Latt\`es map and let $D$ be an isogeny inducing $f$. Then, the dimension in moduli of any family $(f_\lambda)_{\lambda\in\Lambda}$ of Latt\`es maps containing $f$ satisfies $\dim_{\mathcal{M}}(f_\lambda,\Lambda)\leq k(k+1)/2$. More precisely, if $(f_\lambda)_{\lambda\in \Lambda}$ is a maximal family containing $f$, then
\begin{enumerate}
\item either the family is isotrivial, i.e. $\dim_{\mathcal{M}}(f_\lambda,\Lambda)=0$, in which case the eigenspaces of the linear part of $D$ associated with $\sqrt{d}$ and $-\sqrt{d}$ are isotrivial,
\item or $\dim_{\mathcal{M}}(f_\lambda,\Lambda)=q>0$, in which case $\sqrt{d}$ is an integer and the linear part of $D$ has eigenspaces associated with $\sqrt{d}$ and $-\sqrt{d}$ of respective dimensions $q_+$ and $q_-$ with
\[q=(q_++q_-)\cdot\frac{q_++q_-+1}{2}~.\]
\end{enumerate}
 \end{Theorem}

Be aware that this result provides examples of families which are \emph{stable} in the sense of Berteloot-Bianchi-Dupont~\cite{BBD}, since the function $f\mapsto L(f)$, where $L(f)$ is the sum of Lyapunov exponents of $f$ with respect to its maximal entropy measure, is constant equal to $\frac{k}{2}\log d$, hence pluriharmonic on any such family of Latt\`es maps.

In Section 5, we come back to symmetric products and apply Theorem~\ref{tm:rigidLattes} to symmetric product Latt\`es maps. We give a precise description of maximal families of Latt\`es maps on $\p^k$ containing the $k$-fold symmetric products of a rational map. More precisely, we prove the following consequence of Theorem~\ref{tm:rigidLattes}.

\begin{Theorem}
Let $k\geq2$ be an integer and let $d\geq2$ be an integer. Let $F\in\Hom_d(\p^k)$ be the $k$-fold symmetric product of $f\in\Hom_d(\p^1)$. Then
\begin{enumerate}
\item $f$ is rigid if and only if $F$ is rigid,
\item otherwise, $F$ belongs to a family of Latt\`es maps which has dimension $k(k+1)/2$ in moduli. In particular, $F$ can be approximated by Latt\`es maps which are not symmetric products.
\end{enumerate}\label{tm:symmlattes}
\end{Theorem}

We give a more precise result in dimension $2$, relying on the work of Rong~\cite{Rong}. A few words are in order to describe this result.
Rong's result says that (up to taking an appropriate iterate) Latt\`es maps on $\p^2$ are either symmetric products or preserve an algebraic web associated with a cubic curve (see Section~\ref{sec:symLattes} for more details). In turn, our Theorem~\ref{tm:symmlattes} says that any symmetric product Latt\`es map is either rigid, or contained in a family of Latt\`es maps of dimension $3$ in moduli whose symmetric product locus is a \emph{strict} subfamily of dimension $1$ in moduli.

Finally, we give examples of rigid Latt\`es maps of $\p^2$ and $\p^3$. We also provide an example of a non-rigid family of Latt\`es maps of $\p^2$. All three examples rely on Milnor's famous examples of Latt\`es rational maps of degree $2$ and $4$ (see~\cite{Milnor}).

\subsection*{Acknowledgement}
The authors would like to thank AIM for hosting the workshop \emph{post-critically finite maps in complex and arithmetic dynamics} during which we began the present work. We also would like to thank Dylan Thurston for useful discussions. We also thank the referee for many detailed comments which improved the paper.

\subsection*{Outline}
Section \ref{sect_background} describes the construction of basic dynamical properties of symmetric products.  In Section \ref{sect_pcf}, we examine various types of post-critically finite symmetric products and describe when a symmetric product is strongly post-critically finite.  Families of Latt\`es maps in the moduli space $\mathcal{M}_d(\p^k)$ of degree $d$ endomorphisms of $\p^k$ are the focus of Section \ref{sect_lattes}, and this section contains our main result concerning non-isotrivial families of Latt\`es maps.  Finally, in Section \ref{sec:symLattes}, we give examples of rigid and non-rigid Latt\`es maps of $\p^2$.

\section{Dynamics of symmetric products} \label{sect_background}

This section is devoted to a description of the dynamics of symmetric products. Some of the statements of the present section are known and classical, but many of the arithmetic properties are not in the literature.

Let $\eta_k:(\p^1)^k \longrightarrow \p^{k}$ be the holomorphic map defined as the quotient map of the action of symmetric group $\mathfrak{S}_k$ on $(\p^1)^k$ by permuting the terms of the product. In fact, one can write $\eta_k=[\eta_{k,0}\col \cdots \col \eta_{k,k}]$ where $\eta_{k,j}$ is the homogeneous degree $k$ symmetric function given by the elementary homogeneous symmetric polynomials
$$\eta_{k,j}([z_1\col t_1],\ldots,[z_k\col t_k]):=\sum_{\mycom{(I_1,\dots,I_k)\in\{0,1\}^k}{I_1+\cdots+I_k=k-j}}\prod_{\ell=1}^kz_\ell^{I_\ell} \cdot t_\ell^{1-I_\ell}~.$$
Denote by $\omega$ the Fubini-Study form of $\p^1$.

We collect some basic properties of $\eta_k$ in the following proposition.
\begin{proposition} \label{prop1}
The $k$-th symmetric product map $\eta_k:(\p^1)^k \longrightarrow \p^{k}$ is a finite branched cover with degree of branching $k$ and topological degree $k!$.
\end{proposition}

\begin{proof}
 It is shown by Maakestad that there is an isomorphism between $(\mathbb{P}^1)^k$ and $\mathbb{P}^k$ given by the elementary homogeneous symmetric polynomials of degree $k$ \cite[Corollary 2.6]{maakestad}. Indeed, the degree $k$ can be computed as follows: take a hyperplane $H\subset\p^k$, then
    \begin{align*}
    \deg(\eta_k) & =\int_{\eta_k^{-1}(H)}(\omega\boxtimes\cdots\boxtimes\omega)^{k-1}.
    \end{align*}
    Take $H=\{x_{k+1}=0\}$ so that $\eta_k^{-1}(H)=\bigcup_{j=1}^k\{t_j=0\}$. In particular, by Fubini,
    \begin{align*}
    \deg(\eta_k) & =\sum_{j=1}^k\int_{\{t_j=0\}}(\omega\boxtimes\cdots\boxtimes\omega)^{k-1}=\sum_{j=1}^k\left(\int_{\p^1}\omega\right)^{k-1}=k.
    \end{align*}
To compute the topological degree of $\eta_k$, it is sufficient to remark that the action of $\mathfrak{S}_k$ on $(\p^1)^k$ by permutation is free on $\{x=(x_1,\ldots,x_k)\in(\p^1)^k$ with $x_i\neq x_j$ for $i\neq j\}$. In particular, the topological degree of $\eta_k$ is $\mathrm{Card}(\mathfrak{S}_k)=k!$.
\end{proof}

One can actually show that the $k$-fold symmetric product of $f$ is regular and enjoys good dynamical properties with respect to the ones of $f$.

\begin{proposition}
For any $f\in\Hom_d(\p^1)$, the $k$-fold symmetric product $F$ of $f$ is well-defined. Moreover, it satisfies the following properties.
\begin{enumerate}
\item $F$ is an endomorphism of degree $d$ of $\p^{k}$.
\item If $f$ is a degree $d$ polynomial mapping of $\p^1$, then $F$ is a degree $d$ polynomial mapping of $\p^k$, i.e., there is a totally invariant hyperplane.
\end{enumerate}
\end{proposition}

\begin{proof}
\mbox{}
\begin{enumerate}
    \item
Since $f$ can be written $f([z\col w])=[P(z\col w)\col Q(z\col w)]$ in homogeneous coordinates with $P,Q$ homogeneous polynomials, the map $\eta_k\circ (f,\ldots,f):(\p^1)^k\to\p^k$ is invariant under the diagonal action of $\mathfrak{S}_k$. In particular, there exists homogeneous degree $d$ polynomials $P_0,\ldots,P_k$ such that
$$\eta_k\circ(f,\ldots,f)=[P_0\col \cdots \col P_k]\circ \eta_k$$
and the map $F$ exists as a rational mapping of $\p^k$. Now, remark that
$$\bigcap_{j=0}^k\{\eta_{k,j}\circ(f,\ldots,f)=0\}=\emptyset$$
by construction and if $x\in\p^k$ is an indeterminacy point of $F$, it means that $\eta^{-1}_k\{x\}\in\{\eta_{k,j}\circ(f,\ldots,f)=0\}$ for any $0\leq j\leq k$, which is impossible. In particular, the map $F$ has no indeterminacy points.

\item If $f$ is a degree $d$ polynomial, one can write $f([z\col t])=[P(z,t)\col t^d]$, where $P$ is homogeneous of degree $d$. One then can write $\eta_k\circ(f,\ldots,f)([z_1\col t_1],\ldots,[z_k\col t_k])$ as follows:
$$\left[\prod_{\ell=1}^kP(z_\ell,t_\ell)\col \sum_{j=1}^kt_j^d\prod_{\mycom{\ell=1}{ \ell\neq j}}^kP(z_\ell,t_\ell) \col \cdots \col \prod_{\ell=1}^kt_\ell^d\right]~.$$
Since $\eta_{k,k}([z_1\col t_1],\ldots,[z_k \col t_k])=\prod_jt_j$, this gives $P_k(x_0 \col \cdots \col x_k)=x_k^{d}$. Hence, the map $F$ can be written in homogeneous coordinates $$F([x_0\col\cdots \col x_k])=[P_0(x_0,\ldots,x_k)\col \cdots \col P_{k-1}(x_0,\ldots,x_k) \col x_k^d].$$ Hence $F^{-1}\{x_k=0\}\subset \{x_k=0\}$, i.e., the map $F$ is a polynomial endomorphism of $\p^k$.
\end{enumerate}
\end{proof}

Let $F\in\Hom_d(\p^{k})$ be defined over the field of complex numbers. We refer to \cite{Sibony} for a more complete account on the ergodic theory of complex endomorphisms of $\p^{k}$. Let $\mu_F$ be the unique measure of maximal entropy $k\log d$ of $F:\p^k\rightarrow\p^k$. The measure $\mu_F$ is the only probability measure with constant Jacobian, i.e. such that $F^{*}\mu_F=d^{k}\mu_F$, it is invariant, i.e. $F_*\mu_F=\mu_F$, ergodic and mixing and for any $\omega$-psh function $g$ on $\p^{k}$, we have $g\in L^{1}(\mu_F)$. Let also $\pi_j:(\p^1)^k\rightarrow\p^1$ be the projection onto the $j$-th coordinate. Recall that the sum of Lyapunov exponents of $F$ with respect to $\mu_F$ is the real number
\[L(F):=\int_{\p^k}\log|DF|\mu_F~,\]
where $|\cdot|$ denotes any hermitian metric on $\p^k$. The quantity $L(F)$ is independent of the metric, by compactness of $\p^k$ and invariance of $\mu_F$. One of the many consequences of the work of Briend and Duval~\cite{briendduval} is that for any $F\in \Hom_d(\p^k)$, we have $L(F)\geq k\log\sqrt{d}$.

 An easy result is the following.
 We include a proof for the sake of completeness.

\begin{lemma}
Let $F\in\Hom_d(\p^{k})$ be the $k$-fold symmetric product of $f\in\Hom_d(\p^1)$. Then\label{lm:lyap}
\[L(F)=k\cdot L(f)~.\]
\end{lemma}

\begin{proof}
Let $\tilde f:=(f,\ldots,f)$. First, one easily sees that the probability measure
\[\nu_{f,k}:=\bigwedge_{j=1}^k(\pi_j)^*\mu_f\]
 is $\tilde f$-invariant and has constant Jacobian $d^{k}$. As a consequence, $(\eta_k)_*(\nu_{f,k})$ is invariant and has constant Jacobian $d^k$. Since the only probability measure having these properties of the maximal entropy measure of $F$, we get
$$\mu_F=(\eta_k)_*(\nu_{f,k})~.$$
Hence
\begin{eqnarray*}
L(F) & = & \int_{\p^{k}}\log|\det DF|\mu_{F}=\int_{\p^{k}}\log|\det DF|\cdot(\eta_k)_*\left(\nu_{f,k}\right)\\
& = & \int_{(\p^1)^k}\left((\eta_k)^*\log|\det DF|\right)\cdot\nu_{f,k}~.
\end{eqnarray*}
By construction of $F$, the chain rule gives
\begin{eqnarray*}
(\eta_k)^*\log|\det DF| = (\tilde f)^*\log|\det D\eta_k|+\log|\det D\tilde f|-\log|\det D\eta_k|~.
\end{eqnarray*}
As $\nu_{f,k}$ is the intersection of closed positive $(1,1)$-currents with continuous potentials on $(\p^1)^k$, it does not give mass to pluripolar sets. In particular, one has $\log|\det D\eta_k|\in L^1(\nu_{f,k})$ and
\begin{eqnarray*}
L(F) & = & \int_{(\p^1)^k}\log\left|\det D\tilde f\right|\cdot\nu_{f,k} +\int_{(\p^n)^k}\left((\tilde f)^*\log\left|\det D\eta_k\right|-\log\left|\det D\eta_k\right|\right)\cdot\nu_{f,k}\\
& = & \sum_{\ell=1}^k\int_{(\p^1)^k}\log\left|f'\circ \pi_\ell\right|\cdot\nu_{f,k}+\int_{(\p^1)^k}\log\left|\det D\eta_k\right|\cdot\left((\tilde f)_*(\nu_{f,k})-\nu_{f,k}\right).
\end{eqnarray*}
As $\nu_{f,k}=\bigwedge_{j=1}^k(\pi_j)^*\mu_f$ is $(\tilde f)_*$-invariant, we have $(\tilde f)_*(\nu_{f,k})-\nu_{f,k}=0$ and Fubini gives
\begin{eqnarray*}
L(F) & = & \sum_{\ell=1}^k\int_{(\p^1)^k}\log\left|f'\circ \pi_\ell\right|\cdot\nu_{f,k}\\
& = & \sum_{\ell=1}^k\int_{(\p^1)^k}\left(\bigwedge_{\mycom{j=1}{ j\neq \ell}}^k(\pi_j)^*\mu_f\right)\wedge(\pi_\ell)^*\left(\log\left|f'\right|\cdot\mu_f\right)\\
& = & \sum_{\ell=1}^k\left(\int_{\p^1}\mu_f\right)^{k-1}\cdot\int_{\p^1}\log\left|f'\right|\cdot\mu_f\\
& = & \sum_{\ell=1}^k\int_{\p^1}\log\left|f'\right|\cdot\mu_f=k\cdot L(f)~,
\end{eqnarray*}
which ends the proof.
\end{proof}

\section{Post-critically finite symmetric products}\label{sect_pcf}

In this section, we prove Theorem \ref{tm:strongpcf}.  The proof of this result relies on the following description of the critical set of $F$ in terms of critical sets of $f$ and $\eta_k$.

\begin{lemma}
The critical set of $F$ is given by
$$\mathcal{C}(F)=\eta_k\left(\bigcup_{s=0}^{k-1}(\p^1)^s\times\mathcal{C}(f)\times(\p^1)^{k-s-1}\cup\bigcup_{1\leq i<j\leq k}\Omega_{i,j}\right)$$
where $\Omega_{i,j}=\{x\in(\p^1)^k\, ; \ f(x_i)=f(x_j)\}$.
\label{lm:crit}
\end{lemma}

\begin{proof}
Since $F\circ\eta_k=\eta_k\circ(f,\dots,f)$, for all $z\in(\mathbb{P}^{1})^{k}$, we have
\[D_{\eta_k(z)}F\circ D_z\eta_k=D_{(f,\dots,f)(z)}\eta_k\circ D_z(f,\dots,f).\]
Since at any critical point $p$ of $(f,\dots,f)$, there is $\vec v\in T_p(\mathbb P^1)^k$ such that $D_p\eta_k(\vec v)\neq0$, it follows that the critical locus of $F$ is the projection by $\eta_k$ of the critical locus of the map $(f,\ldots,f)$ acting on $(\p^1)^k$ and of the critical locus of $\eta_k\circ(f,\dots,f)$.  The critical locus of $(f,\ldots,f)$ obviously is
$$\bigcup_{s=0}^{k-1}(\p^1)^s\times\mathcal{C}(f)\times(\p^1)^{k-s-1},$$
and the critical locus of $\eta_k\circ(f,\dots,f)$ is the set where the action of the group $\mathfrak{S}_k$ is not free, i.e. $(x_1,\ldots,x_k)\in(\p^1)^k$ is in the critical set of $\eta_k\circ (f,\dots,f)$ if and only if $f(x_i)=f(x_j)$ for some $1\leq i<j\leq k$.

Thus, it remains only to show that at any critical point $p$ of $(f,\dots,f)$, there is $\vec v\in T_p(\mathbb P^1)^k$ such that $D_p\eta_k(\vec v)\neq0$.  Without loss of generality, we may use the affine coordinates $w_i=z_i/t_i$ for $1\leq i\leq k$ in the domain so that
\[\eta_{k,j}(w_1,\dots,w_k)=\sum_{\mycom{(I_1,\dots,I_k)\in\{0,1\}^k}{I_1+\cdots+I_k=k-j}}\prod_{\ell=1}^kw_\ell^{I_\ell}.\]
In particular, we have
\[\eta_{k,k-1}(w_1,\dots,w_k)=w_1+\cdots+w_k\quad\mbox{and}\quad\eta_{k,k}(w_1,\dots,w_k)=1.\]
Using coordinates $\eta_{k,j}/\eta_{k,k}$ for $0\leq j\leq k-1$ in the codomain, we have $\eta_k\colon(\mathbb C)^k\rightarrow\mathbb C^k$.  Since $\eta_{k,k-1}(w_1,\dots,w_k)/\eta_{k,k}(w_1,\dots,w_k)=w_1+\cdots+w_k$, every entry in the last row of the matrix representing $D_p\eta_k$ is $1$, whence there exists $\vec v\in T_p(\mathbb C)^k\setminus\ker(D_p\eta_k)$, as required.
\end{proof}

\begin{proof}[Proof of Theorem \ref{tm:strongpcf}]
Let us first prove that $F$ is post-critically finite if and only if $f$ is post-critically finite. Remark that the post-critical set of $F$ is an algebraic variety if and only if it is the image under $\eta_k$ of an algebraic variety. According to Lemma \ref{lm:crit}, one thus sees that $F$ is post-critically finite if and only if the union of the iterates under the map $(f,\ldots,f)$ (resp. $(f,f)$) of the set
$$\bigcup_{s=0}^{k-1}(\p^1)^s\times\mathcal{C}(f)\times(\p^1)^{k-s-1}\subset(\p^1)^k$$
(resp. $\Omega=\{(x_1,x_2)\in(\p^{1})^{2}\, ; / f(x_1)=f(x_2)\}\subset(\p^1)^2$), which is
$$\bigcup_{s=0}^{k-1}(\p^1)^s\times\mathcal{P}(f)\times(\p^1)^{k-s-1}\subset(\p^1)^k$$
(resp. the diagonal $\Delta$ of $(\p^1)^2$)  is an algebraic variety of $(\p^1)^k$ (resp. $(\p^1)^2$). Hence, $F$ is post-critically finite if and only if $f$ is.

Let us now prove that $F$ is strongly post-critically finite if an only if $f$ is post-critically finite. Write $\Delta_{i,j}:=\{x\in(\p^{1})^{k}\, ; \ x_i=x_j\}$ and note that $(f,\ldots,f)(\Omega_{i,j})=\Delta_{i,j}$ and that $\Delta_{i,j}$ is fixed by $(f,\ldots,f)$. By Lemma \ref{lm:crit}, the problem reduces to proving that $(f,\ldots,f)$ is strongly post-critically finite as a self-map of $(\p^1)^k$ and as a self-map of $\bigcup_{i<j}\Delta_{i,j}$ if and only if $f$ is post-critically finite. Let us proceed by induction on $k\geq2$. For $k=2$, by Lemma \ref{lm:crit}, the critical set of $(f,f)$ is $\p^1\times\mathcal{C}(f)\cup\mathcal{C}(f)\times\p^1$ and the critical set of its restriction to $\p^1\times\mathcal{C}(f)\cup\mathcal{C}(f)\times\p^1$ is $\mathcal{C}(f)\times\mathcal{C}(f)$. Moreover, the critical locus of the restriction of $(f,f)$ to $\Delta$ is also $\mathcal{C}(f)\times \mathcal{C}(f)$. Moreover, $f$ is post-critically finite if and only if the restriction of $(f,f)$ to $\p^1\times\mathcal{C}(f)\cup\mathcal{C}(f)\times\p^1\cup \Delta$ is strongly post-critically finite, which concludes the proof for $k=2$.

Assume now that for all $2\leq s< k$, the $s$-fold symmetric product of $f$ is strongly post-critically finite if and only if $f$ is post-critically finite. According to Lemma \ref{lm:crit} and to the above discussion, it is sufficient to prove that $(f,\ldots,f)$ restricted to
$$\bigcup_{s=0}^{k-1}(\p^1)^s\times\mathcal{C}(f)\times(\p^1)^{k-s-1}\cup\bigcup_{1\leq i<j\leq k}\Delta_{i,j}$$
is strongly post-critically finite if and only if $f$ is post-critically finite. By the induction assumption, the strong post-critical finiteness of $(f\ldots,f)$ on $(\p^1)^s\times\mathcal{C}(f)\times(\p^1)^{k-s-1}$ is equivalent to that of $f$. Finally, there exists a map $\sigma_{i,j}:(\p^1)^k\to(\p^1)^k$ permuting coordinates such that $\sigma_{i,j}(\Delta_{i,j})=\Delta\times (\p^1)^{k-2}$, and the following diagram commutes
$$\xymatrix {\relax
\Delta_{i,j} \ar[r]^{(f,\ldots,f)} \ar[d]_{\sigma_{i,j}} & \Delta_{i,j}\ar[d]^{\sigma_{i,j}} \\
\Delta\times (\p^1)^{k-2} \ar[r]_{(f,\ldots,f)} & \Delta\times (\p^1)^{k-2}}$$
In particular, the problem reduces to the case of the $(k-2)$-fold symmetric product of $f$ and the induction assumption allows us to conclude.
\end{proof}

We now want to give a negative answer to the question in the case when indeterminacy points are allowed. To do so, we provide an example based on a map from \cite{kaschner} of a degree $4$ rational mapping of $\p^2$ having the wanted properties. For any $a\in\C\setminus\{0\}$, we let
\[F_a\colon\p^2\dashrightarrow\p^2\quad\mbox{by}\quad[x:y:z]\mapsto[-y^2 : ax^2-axz : z^2-x^2].\]
%$$\xymatrix {\relax F_a:[x:y:z]\in\p^2\ar@{-->}[r] & [-y^2 : ax^2-axz : z^2-x^2]\in\p^2~.}$$
This provides a good counter-example.
\begin{proposition}
The following holds:
\begin{enumerate}
\item For all $a\in\C\setminus\{0\}$, the map $F_a^2$ is rational (i.e. $\cal{I}(F_a^2)\ne\emptyset$) and post-critically finite.
\item For all but countably many $a\in\C\setminus\{0\}$, the map $F_a^2$ is not strongly post-critically finite.
\end{enumerate}
\end{proposition}

\begin{proof}
For any $a\neq0$, the point $I=[1:0:1]$ is clearly the unique indeterminacy point of $F_a$. It also is an indeterminacy point of $F_a^2$, hence $F_a^2$ is rational. Moreover, the critical locus of $F_a$ is
\[\cal{C}(F_a)=\{x=0\}\cup\{y=0\}\cup\{z=x\}~.\]
Note that $I\in\{y=0\}\subset\mathcal C(F_a)$ and $F_a(I)=\{x=0\}$.   Since $F_a(\{y=0\})=\{x=0\}$, $F_a(\{x=0\})= \{y=0\}$, and $F_a(\{x=z\})=\{[1:0:0]\}$, we get
\[F_a(\mathcal{C}(F_a))\subset\cal{C}(F_a)\setminus\{I\}\]
and $\mathcal{P}(F_a)=\{x=0\}\cup\{y=0\}$ is disjoint from $I$. Hence, $F_a$ is post-critically finite, as is $F_a^2$. To prove that for a good choice of $a\neq0$ the map $F_a^2$ is not strongly post-critically finite,
we look at the restriction
\[f_a:=F^2_a|_{\{y=0\}}:\{y=0\}\simeq\p^1\rightarrow\{y=0\}\simeq\p^1,\]
which is the quadratic rational map of $\p^1$:
\[f_a([z:t])=[-a^2z^2(z-t)^2:(z^2-t^2)^2]=[-a^2z^2:(z+t)^2],\]
$[z:t]\in\p^1$. Its critical points are $[0:1]$ and $[1:1]$. Moreover, $f_a([0:1])=[0:1]$, hence, $f_a$ is a quadratic polynomial.  It is conjugate to the map
\[p_a:z\in\C\mapsto z^2-\frac{1}{a^2}\in\C~.\]
In particular, $f_a$ is post-critically finite if and only if $p_a$ is post-critically finite, i.e. if and only if $p_a^n(0)=p_a^k(0)$ for some $n>k\geq0$. But the family $(p_a)_{a\in\C\setminus\{0\}}$ is a $2$-to-$1$ cover of a punctured line of $\mathcal{M}_2(\p^1)\setminus\{[z^2]\}$ so no equation of the form $p_a^n(0)=p_a^k(0)$ is satisfied by all $a\neq0$. In particular, there exists only countably many such parameters.
\end{proof}

 \section{Families of Latt\`es maps in dimension at least 2} \label{sect_lattes}
 Our main goal in this section is to prove Theorem \ref{tm:rigidLattes}; the proof decomposes into two steps. First, we prove that the induced analytic family of complex dimension $k$ tori up to biholomorphism has the same dimension as the family of Latt\`es maps. We then follow Milnor's proof for the case $k=1$ and adapt his argument.

\subsection{Families of abelian varieties}

This paragraph is devoted to recalling classical results concerning abelian varieties. For the material of this paragraph, we refer to~\cite{Debarre}.

\medskip

\par\noindent\textit{The moduli space of polarized abelian varieties.}

%\par A \emph{polarization} on a complex abelian variety $\mathcal{T}$ is an \emph{integral} K\"ahler form $\omega$ on $\mathcal{T}$, i.e. a K\"ahler form with integral cohomology class in $H^2(\mathcal{T},\mathbb{R})$. To $(\mathcal{T},\omega)$ can be associated integers $d_1|\cdots|d_k$ and a matrix $\tau\in M_k(\C)$ with $\tau^t=\tau$ and $\mathrm{Im}(\tau)>0$ such that $\mathcal{T}=\C^k/\Gamma_\tau$, with $\Gamma_\tau:=\tau \Z^k\oplus\Delta\Z^k$ and $\Delta=\mathrm{diag}(d_1,\ldots,d_k)\in M_k(\Z)$. We say $\mathcal{T}$ is of \emph{type} $\Delta$.
\par A \emph{polarization} on a complex abelian variety $\mathcal{T}$ is an \emph{integral} K\"ahler form $\omega$ on $\mathcal{T}$, i.e. a K\"ahler form with integral cohomology class in $H^2(\mathcal{T},\mathbb{R})$.  A complex torus $\mathcal T=\mathbb C^k/\Gamma$ admits an integral K\"ahler form if and only if there is a basis for $\mathbb C^k$, integers $d_1,d_2,\dots,d_k$ satisfying $d_1\,|\cdots|\,d_k$, and a matrix $\tau\in M_k(\C)$ with $\tau^t=\tau$ and $\mathrm{Im}(\tau)>0$ such that $\Gamma_\tau:=\tau \Z^k\oplus\Delta\Z^k$, where $\Delta=\mathrm{diag}(d_1,\ldots,d_k)\in M_k(\Z)$.  In this case, we say the polarization $\omega$ on $\mathcal{T}$ is of \emph{type} $\Delta$.
%There is a basis for the lattice $\Gamma_\tau$ in which the matrix of $\omega$ is
%\[\left(\begin{array}{cc}0&\Delta\\-\Delta&0\end{array}\right).\]

We now give the classical definition of the moduli space of type $\Delta$ abelian varieties. Let
\[\mathcal{H}_k:=\{\tau \in M_k(\C)\, ; \ \tau^t=\tau ~, \ \mathrm{Im}(\tau)>0\}.\]
For any field $K$, let $\mathrm{Sp}_{2k}(K)$ be the symplectic group
\[\mathrm{Sp}_{2k}(K):=\{ M\in \mathrm{GL}_{2k}(K)\, ; \ MJM^t=J\}~,\]
where $J=\left(\begin{array}{cc}
0 & I_k\\
-I_k & 0
\end{array}\right)\in M_{2k}(K)$.
We also let $\sigma_\Delta:\mathrm{GL}_{2k}(\Q)\longrightarrow \mathrm{GL}_{2k}(\Q)$ be defined by
\[\sigma_\Delta(M):=\left(\begin{array}{cc}
I_k & 0 \\
0 & \Delta \\
\end{array}\right)^{-1}M\left(\begin{array}{cc}
I_k & 0 \\
0 & \Delta \\
\end{array}\right), \ \ M\in \mathrm{GL}_{2k}(\Q).\]
Finally, we let $G_\Delta:=\sigma_\Delta(M_{2k}(\Z))\cap\mathrm{Sp}_{2k}(\Q)$.

~

The well-defined map
\[\left(M=\left(\begin{array}{cc}
A & B\\
C & D
\end{array}\right),\tau\right) \longmapsto M\cdot\tau:=(A\tau+B)(C\tau+D)^{-1}\]
defines an action of the group $\mathrm{Sp}_{2k}(\R)$ on $\mathcal{H}_k$. Moreover, any discrete subgroup of $\mathrm{Sp}_{2k}(\R)$ acts properly discontinuously on $\mathcal{H}_k$. The main result we rely on is that two polarized abelian varieties $\C^k/\Gamma_\tau$ and $\C^k/\Gamma_\gamma$ of type $\Delta$ are isomorphic if and only if there exists $M\in G_\Delta$ with $M\cdot\tau=\gamma$.

\begin{definition}
The \emph{moduli space} $\mathcal{A}_{k,\Delta}$ of polarized abelian varieties of type $\Delta$ is the set of isomorphism classes of such varieties. Equivalently, this is $\mathcal{H}_k/G_\Delta$.
\end{definition}

It is known that the moduli space $\mathcal{A}_{k,\Delta}$ is a quasi-projective variety and that it has dimension $k(k+1)/2$. We also denote by
\[\Pi_\Delta:\mathcal{H}_k\longrightarrow \mathcal{A}_{k,\Delta}\]
the quotient map.

\medskip

\paragraph*{\textit{Dimension in moduli of families of abelian varieties.}}
We say that $(\mathcal{T}_t)_{t\in \Lambda}$ is a \emph{holomorphic family} of abelian varieties if there exists a complex manifold $\widehat{\mathcal{T}}$ and a holomorphic map $p:\widehat{\mathcal{T}}\longrightarrow\Lambda$, such that $p^{-1}\{t\}\cong\mathcal{T}_t$ for all $t\in\Lambda$.

We want to define properly a notion of dimension in moduli for the family $(\mathcal{T}_t)_{t\in\Lambda}$.
First, notice that the polarization of $\mathcal{T}_t$ is defined by an \emph{entire} K\"ahler form. As seen in the proof of~\cite[Proposition VI.1.3]{Debarre}, when $t$ varies in $\Lambda$, the integers $d_1\,|\cdots|\,d_k$ cannot change. In particular, the type of the polarization is constant in the family $(\mathcal{T}_t)_{t\in\Lambda}$ and there exists a holomorphic map $t\in\Lambda\longmapsto \tau(t)\in\mathcal{H}_k$ with $\mathcal{T}_t=\C^k/(\tau(t)\Z^k\oplus\Delta \Z^k)$ for all $t\in\mathcal{H}_k$.

As a consequence, we can define the \emph{type} of a holomorphic family $(\mathcal{T}_t)_{t\in\Lambda}$ of abelian varieties as the type $\mathcal{T}_{t_0}$ of any $t_0\in\Lambda$.

\begin{definition}
We say that a holomorphic family $(\mathcal{T}_t)_{t\in\Lambda}$ of abelian varieties of type $\Delta$ has \emph{dimension $q$ in moduli} if the analytic set $\Pi_\Delta(\Lambda)\subset\mathcal{A}_{k,\Delta}$ has dimension $q$.
\end{definition}

 \subsection{Latt\`es maps and abelian varieties}

\paragraph*{\textit{Families of Latt\`es maps.}}

Recall that $F\in\Hom_d(\p^k)$ is a \emph{Latt\`es map} if there exists a complex $k$-dimensional abelian variety $\mathcal{T}$, an isogeny $\mathcal{I}:\mathcal{T}\to \mathcal{T}$, and a Galois branched cover $\Theta:\mathcal{T}\to\p^k$ making the following diagram commute
\begin{center}
$\xymatrix {\relax
\mathcal{T} \ar[r]^{\mathcal{I}} \ar[d]_{\Theta} & \mathcal{T}\ar[d]^{\Theta} \\
\p^{k} \ar[r]_{F} & \p^{k}}$
\end{center}
(see e.g. \cite{Dupont-spherical}). These maps are known to be post-critically finite (see e.g. \cite{bertelootdupont}). Another way to present Latt\`es maps, given by Berteloot and Loeb~\cite[Th\'eor\`eme 1.1 $\&$ Proposition 4.1]{BertelootLoeb}, is the following (see also~\cite{bertelootdupont}):  A \emph{complex crystallographic} group $G$ is a discrete group of affine transformations of a complex affine space $V$ such that the quotient $X = V/G$ is compact.

A degree $d$ endomorphism $F:\p^k\longrightarrow\p^k$ is a Latt\`es map if there exists a ramified cover $\sigma:\C^k\longrightarrow\p^k$, an affine map $A:\C^k\longrightarrow\C^k$ whose linear part is $\sqrt{d}\cdot U$ where $U\in\mathbb{U}(k)$ is a diagonal unitary linear map, and a complex crystallographic group $G < \mathbb{U}(k)\rtimes\C^k$ (hence discrete and closed) such that the following diagram commutes:
 \begin{center}
$\xymatrix {\relax
\C^k \ar[r]^{A} \ar[d]_{\sigma} & \C^k\ar[d]^{\sigma} \\
\p^k \ar[r]_{F} & \p^k }$
\end{center}
and the group $G$ acts transitively on fibres of $\sigma$.

 \begin{lemma}\label{lm:unitary}
Let $(f_t)_{t\in\Lambda}$ be any holomorphic family of Latt\`es maps of $\p^k$. Then there exists a holomorphic family $(G_t)_{t\in\Lambda}$ of complex crystallographic groups, a holomorphic family $(\sigma_t)_{t\in\Lambda}$ of branched covers, and a holomorphic family of affine maps $(A_t)_{t\in\Lambda}$, as above. Moreover, the linear part of $A_t$ is independent of $t$.
 \end{lemma}

 \begin{proof}
 Let $(f_t)_{t\in\Lambda}$ be a holomorphic family of Latt\`es maps and let $G_t$, $A_t$ and $\sigma_t$ be such that $\sigma_t\circ A_t=f_t\circ\sigma_t$ on $\C^k$ be given as in~\cite{BertelootLoeb}. It follows from Sections 4 and 5 of~\cite{BertelootLoeb} that $G_t$, $A_t$ and $\sigma_t$ depend holomorphically on $t$. Moreover, the linear part of $A_t$ can be written $\sqrt{d}\cdot U_t$ with $U_t\in\mathbb{U}(k)\subset \mathrm{GL}_k(\C)$ and the map $t\mapsto U_t$ is holomorphic. Hence, it has to be constant.
  \end{proof}

\paragraph*{\textit{Equality of the dimensions in moduli.}}

We now may prove the following key fact.
\begin{proposition}\label{prop_samedim}
Let $\Lambda$ be a complex manifold. Let $(f_\lambda)_{\lambda\in\Lambda}$ be a holomorphic family of degree $d$ Latt\`es maps of $\p^k$, and let $(\mathcal{T}_\lambda)_{\lambda\in\Lambda}$ be any induced family of abelian varieties. Then the family $(f_\lambda)_{\lambda\in\Lambda}$ has dimension $q$ in moduli if and only if $(\mathcal{T}_\lambda)_{\lambda\in\Lambda}$ has dimension $q$ in moduli.
\end{proposition}

\begin{proof}
Pick any induced holomorphic family of tori $(\mathcal{T}_t)_{t\in\Lambda}$. Let $q$ and $r$ stand, respectively, for the dimension in moduli of $(f_t)_{t\in\Lambda}$ and $(\mathcal{T}_t)_{t\in\Lambda}$. Assume first that $r>q$. Choose $t_0\in \Lambda$. By assumption, there exists a holomorphic disk $\D\subset\Lambda$ centered at $t_0$ and such that $(\mathcal{T}_t)_{t\in \D}$ has dimension $1$ in moduli and the canonical projection $\Pi:\D\longmapsto\mathcal{M}_d(\p^k)$ is constant. In particular, there exists a holomorphic disk $(\phi_t)_{t\in\D}$ of $PSL(k+1,\C)$ such that $f_t=\phi_t^{-1}\circ f_{t_0}\circ\phi_t$ on $\p^k$. Since $(\mathcal{T}_t)_t$ is an induced family of abelian varieties, we have
\begin{center}
$\xymatrix {\relax
\mathcal{T}_t \ar[r]^{D_t} \ar[d]_{\phi_t\circ\Theta_t} & \mathcal{T}_t\ar[d]^{\phi_t\circ\Theta_t} \\
\p^k \ar[r]_{f_{t_0}} & \p^k~. }$
\end{center}
As a consequence, $\phi_t\circ\Theta_t:\mathcal{T}_t\rightarrow\p^k$ and $\Theta_{t_0}:\mathcal{T}_{t_0}\rightarrow\p^k$ are isomorphic Galois branched covers. Hence, there exists an analytic isomorphism $\psi_t:\mathcal{T}_t\rightarrow\mathcal{T}_{t_0}$ for any $t\in\D$. Hence, $\mathcal{T}_t$ is isomorphic to $\mathcal{T}_{t_0}$ for any $t\in\D$. This implies that $(\mathcal{T}_t)_{t\in\D}$ is isotrivial in moduli. This is a contradiction.

~

For the converse inequality, we also proceed by contradiction. Assume $r<q$. Since $\mathcal{M}_d(\p^k)$ is a geometric quotient, for any $t_0\in\Lambda$, there exists a local dimension $q$ complex submanifold $\Lambda_0\subset \Lambda$ containing $t_0$ and such that the canonical projection $\Pi:\Lambda_0\longmapsto \mathcal{M}_d(\p^k)$ has discrete fibers over its image.

As above, this implies the existence of a holomorphic disk $\D\subset\Lambda_0$ centered at $t_0$ and such that the corresponding family $(\mathcal{T}_t)_{t\in\D}$ is isotrivial in moduli, i.e. $\mathcal{T}_t\simeq\mathcal{T}_{t_0}$ as abelian varieties for any $t\in\D$. We, thus, may assume that $\mathcal{T}:=\mathcal{T}_{t_0}=\mathcal{T}_t$ for all $t\in\D$. We, thus, have a holomorphic family $(D_t)_{t\in\D}$ of degree $\sqrt{d}$ isogenies of $\mathcal{T}$ and a holomorphic family of Galois covers $\Theta_t:\mathcal{T}\longrightarrow\p^k$ with
\begin{center}
$\xymatrix {\relax
\mathcal{T} \ar[r]^{D_t} \ar[d]_{\Theta_t} & \mathcal{T}\ar[d]^{\Theta_t} \\
\p^k \ar[r]_{f_t} & \p^k~. }$
\end{center}
Let $\Gamma$ be a lattice defining $\mathcal{T}$, i.e. such that $\mathcal{T}=\C^k/\Gamma$. Lifting $D_t$ to an affine map $\tilde D_t:\C^k\longrightarrow\C^k$, we end up with a holomorphic map $t\in\D\longmapsto\tilde D_t\in\textup{Aff}(\C^k)$, which satisfies $\tilde D_t(\Gamma)\subset\Gamma$ and $\tilde D_t=\sqrt{d}\cdot U_t+\tau_t$ with $U_t\in\mathbb{U}_k$ and $\tau_t\in\frac{1}{2}\Gamma$. As a consequence, the map $t\mapsto\tilde D_t$ is discrete, hence constant. The family $(D_t)_{t\in\D}$ is thus constant. Let $D:=D_t$ for all $t$.

Finally, by assumption, there exists an isomorphism $\alpha_t:\mathcal{T}_t\rightarrow\mathcal{T}_{t_0}$ for all $t\in\D$, which depends analytically on $t$. Hence $\Theta_t=\Theta_{t_0}\circ \alpha_t$ for all $t\in\D$, and $(\Theta_t)_{t\in\D}$ is a holomorphic family of Galois cover of $\p^k$ which are isomorphic.
As a consequence, there exists a holomorphic disk $t\in\D\longmapsto\phi_t\in PSL(k+1,\C)$ such that $\Theta_t=\phi_t\circ \Theta_{t_0}$, for all $t\in\D$. Hence, we get
\[\phi_t\circ\left(\Theta_{t_0}\circ D\right)=\left(f_t\circ\phi_t\right)\circ \Theta_{t_0}\]
which we may rewrite
\[f_{t_0}\circ \Theta_{t_0}=\Theta_{t_0}\circ D=\left(\phi_t^{-1}\circ f_t\circ\phi_t\right)\circ \Theta_{t_0}~.\]
This gives $f_{t_0}=\phi_t^{-1}\circ f_t\circ \phi_t$, i.e. $\D\subset\Pi^{-1}\{f_{t_0}\}$. This is a contradiction since the fibers of $\Pi$ are discrete  in $\Lambda_0$.
\end{proof}

An immediate consequence is the following.
\begin{corollary}\label{cor_samedim}
Let $(f_\lambda)_{\lambda\in\Lambda}$ be a holomorphic family of degree $d$ Latt\`es maps of $\p^k$. Then the family $(f_\lambda)_{\lambda\in\Lambda}$ has dimension in moduli at most $k(k+1)/2$.
\end{corollary}

\subsection{Classification of maximal families of Latt\`es maps}

 We now want to prove Theorem~\ref{tm:rigidLattes}. We use the description of Berteloot and Loeb of Latt\`es maps \cite{BertelootLoeb}. When $U\colon\C^k\rightarrow\C^k$ is linear, we denote by $E_\lambda:=\ker\left(\lambda I_k-U\right)$ the \emph{eigenspace} associated with the eigenvalue $\lambda\in\C$.

 We may prove the following.

\begin{theorem}\label{tm:rigidlattesnewversion}
Assume $(f_t)_{t\in\Lambda}$ is a maximal family of Latt\`es maps containing $f$ and let $D=\sqrt{d}\cdot U+\gamma$ be the affine map inducing $f$. Let $q_{\pm}:=\dim_\C \ker(U\pm I_k)$; then
\[\dim_{\mathcal{M}}(f_t,\Lambda)=\left(q_{+}+q_{-}\right)\cdot\frac{q_++q_{-}+1}{2}.\]
\end{theorem}

\begin{remark}
Let $(f_t)_{t\in\Lambda}$ be any holomorphic family of Latt\`es maps of $\p^k$, let $(A_t)_{t\in\Lambda}$ be the holomorphic family of affine maps given by Lemma~\ref{lm:unitary}, and write $\sqrt{d}\cdot U$ the linear part of $A_t$ with $U\in\mathbb{U}(k)$. Be aware that for any integer $m\in\Z$, since $U$ is unitary, we have $\dim_\C E_{m}(\sqrt{d}\cdot U)=0$ for all $m\neq\pm\sqrt{d}$ and $\dim_\C E_{m}(\sqrt{d}\cdot U)=0$ for all $m$ if $d$ is not the square on an integer.
\end{remark}

\begin{proof}
Up to taking a connected subvariety of $\Lambda$, we may assume $q:=\dim_{\mathcal{M}} (f_t,\Lambda)$. According to Proposition~\ref{prop_samedim}, the induced family of abelian varieties $(\mathcal{T}_t)_{t\in \Lambda}$ has also dimension $q$ in moduli.

Choose $t_0\in\Lambda$. Notice that, by Lemma~\ref{lm:unitary}, if $A_t:\mathcal{T}_t\rightarrow\mathcal{T}_t$ is the affine map inducing $f_t$, then $A_t$ for all $t$ has linear part $U\in\mathbb{U}(k)$ independent of $t$. Our assumption implies that the isogeny $A=(A_t)_t$ preserves a family of abelian varieties with dimension in moduli $q$. Set
\[q_*:=\left(q_{+}+q_{-}\right)\cdot\frac{q_{+}+q_{-}+1}{2}~.\]
We now may prove it implies $q=q_*$. We prove first $q\geq q_*$. If $q_*=0$, this is trivial. Otherwise, up to linear change of coordinates, we have
\[U=\left(\begin{array}{ccc}
I_{q_{+}} & 0 & 0\\
0 & I_{q_{-}} & 0\\
0 & 0 & \star
\end{array}\right)\]
where $\star$ is a $(k-(q_++q_-))$-square diagonal unitary matrix. Let $\Gamma_{\tau_0}=\tau_0\Z^k\oplus\Delta\Z^k$ be a lattice preserved by $A_{t_0}$ with $\mathcal{T}_{t_0}\simeq\C^k/\Gamma_{\tau_0}$. Write
\[\tau_0=\left(\begin{array}{cc}
\tau_{1,1} & \tau_{1,2}  \\
\tau_{1,2}^t & \tau_{2,2}
\end{array}\right)\]
where $\tau_{1,1}\in\mathcal{H}_{q_*}$, $\tau_{2,2}\in\mathcal{H}_{k-q_*}$ and $\tau_{1,2}$ is a $(q_*,k-q_*)$-matrix.
For $\lambda\in\mathcal{H}_{q_*}$, we let
\[\tau(\lambda):=\left(\begin{array}{cc}
\lambda & \tau_{1,2}  \\
\tau_{1,2}^t & \tau_{2,2}
\end{array}\right)\in\mathcal{H}_{k}.\]
Let $\mathcal{T}_{\lambda}:=\C^k/\Gamma_{\lambda}$, where $\Gamma_{\lambda}=\tau(\lambda)\mathbb Z^k\oplus\Delta\mathbb Z^k$. The family $(\mathcal{T}_\lambda)$ has dimension $q_*$ in moduli. Remark also that
\[\gamma=\tau_0\cdot \underline{u}+\Delta\cdot \underline{v}\in\alpha\cdot  \Gamma_{0}\]
for some $\alpha\in\C^*$ (resp. $\underline{u},\underline{v}\in \alpha\cdot \Z^k$) and that this property has to be preserved in any family of Latt\`es maps containing $f$. Hence, we may let
\[\gamma(\lambda):=\tau(\lambda)\cdot \underline{u}+\Delta\cdot \underline{v}\in\alpha\cdot  \Gamma_{\lambda}~, \ \lambda\in\Lambda~,\]
so that the map $\gamma:\Lambda\rightarrow\C^k$ is holomorphic and the map $A_\lambda:=\sqrt{d}\cdot U+\gamma(\lambda)$ induces an isogeny of the abelian variety $\mathcal{T}_\lambda$. Moreover, we have $\mathcal{T}_{\tau_{1,1}}\simeq \mathcal{T}_{t_0}$ and $f_{t_0}$ is induced by $A_{\tau_0}$. Hence, $f_{t_0}$ belongs to a family of Latt\`es maps which is parameterized by $\Lambda$. By Proposition~\ref{prop_samedim}, this family has dimension at least $q_*$ in moduli, and $q\geq q_*$.

~

We now prove the converse inequality by contradiction, assuming $q>q_*$. Assume first that $q_*=0$. Since $q>0$, we again can find a holomorphic disk $\D\subset\Lambda$ with $\dim\Pi_\Delta(\D)=1$ where $\Delta$ is the type of the family $(\mathcal{T}_t)_{t\in D}$. By assumption, $U\in\mathbb{U}(k)$ is diagonal and has no integer eigenvalue. Hence, $U=\mathrm{diag}(u_1,\ldots,u_k)$ with $|u_i|=1$ and $u_i^2\neq1$. On the other hand, if $\sigma_t:\C^k\rightarrow\p^k$ is the Galois cover making the following diagram commute
\begin{center}
$\xymatrix {\relax
\C^k\ar[r]^{A_t} \ar[d]_{\sigma_t} & \C^k\ar[d]^{\sigma_t} \\
\p^k \ar[r]_{f_t} & \p^k~,}$
\end{center}
for any $t\in \D$, any $z\in\C^k$, and any $a\in\Gamma_t$, we find
\[\sigma_t(\sqrt{d}\cdot U(z)+\gamma_t)=f_t(\sigma_t(z))=f_t(\sigma_t(z+a))=\sigma_t(\sqrt{d}\cdot U(z+a)+\gamma_t).\]
This gives $\sqrt{d}\cdot U(\Gamma_t)\subset\Gamma_t$ for all $t$. In particular, if $(e_1,\ldots,e_k)$ is the canonical basis of $\C^k$
\[\sqrt{d}\cdot ( U\circ \tau_t)\cdot e_i=\tau_t\cdot \underline{u}+\Delta\cdot \underline{v}\]
for some $\underline{u},\underline{v}\in\Z^k$, and $\tau_t$ is a non-constant holomorphic map of the parameter $t$ by assumption.
Since $\sqrt{d}\cdot U$ is diagonal, this gives
\[\sqrt{d}\cdot u_i \sum_{j=1}^k\tau_{t,i,j}\delta_{i,j}=\sum_{j=1}^k\tau_{t,i,j}u_j+\Delta_{i,j}v_j~.\]
This is impossible, since it gives that $\tau_t$ is constant on $\D$ hence that $(\mathcal{T}_t)$ is an isotrivial family.

Assume finally that $q_*>0$. As in the proof of Proposition \ref{prop_samedim}, since $\mathcal{M}_d(\p^k)$ is a geometric quotient, for any $t_0\in\Lambda$ there exists a local dimension $q$ complex submanifold $\Lambda_0\subset\Lambda$ containing $t_0$ and such that the canonical projection $\Pi:\Lambda_0\longmapsto \mathcal{M}_d(\p^k)$ has discrete fibers over its image.

Notice that for any $t_0\in X$, there exists a polydisk $\D^{q_*}$ centered at $t_0$ in $X_0$, as described above. By assumption, there exists a disk $\D$ centered at $t_0$ which is transverse to $\D^{q_*}$.
As above, write $\Gamma_t=\tau(t)\Z^k\oplus\Delta\Z^k$ and decompose $\tau(t)$:
\[\tau(t):=\left(\begin{array}{cc}
\tau_{1,1}(t) & \tau_{1,2}(t)  \\
\tau_{1,2}(t)^t & \tau_{2,2}(t)
\end{array}\right)~,\]
where $\tau_{i,j}(t)$ is holomorphic on $t\in \D$. Let $\Gamma'_t:=\tau_{2,2}(t)\Z^k\oplus\Delta'\Z^k$, where
\[\Delta'=\mathrm{diag}(d_{k-q_*+1},\ldots,d_k)\]
and the $d_i$'s are given by $\Delta=\mathrm{diag}(d_1,\ldots,d_k)$ with $d_i\mid d_{i+1}$.
Up to reducing $\D$, we may assume $\Gamma_t'$ is not $G_{\Delta'}$-equivalent to $\Gamma_{t_0}'$ and that $\Delta'$ is the type of the family $(\mathcal{T}_t')_{t\in \D}$ of abelian varieties defined as $\mathcal{T}_t':=\C^{k-q_*}/\Gamma_t'$. We now let $U'=\mathrm{diag}(u_{k-q_*+1},\ldots,u_k)\in\mathbb{U}(k-q_*)$ where, again, the $u_i$'s are given by $U=\mathrm{diag}(u_1,\ldots,u_k)$. We finally let $\gamma_t'=(\gamma_{t,k-q_*+1},\ldots,\gamma_{t,k})\in\C^{k-q_*}$ and $A'_t:=\sqrt{d}\cdot U'+\gamma_t'$. The family $(\mathcal{T}_t',A_t)_t$ induces a family of Latt\`es maps on $\p^{k-q_*}$ with $\dim(U'\mp I_{k-q_*})=0$. We thus have reduced to the case $q_*=0$.
\end{proof}

An immediate corollary is

\begin{corollary}
The dimension in moduli of $(f_t)_{t\in \Lambda}$ is non-zero if and only if $d$ is the square of an integer and the linear map $U$ has at least one integer eigenvalue (which must be $-1$ or $1$).
\end{corollary}

We now derive Theorem~\ref{tm:rigidLattes} from Theorem~\ref{tm:rigidlattesnewversion}.
\begin{proof}[Proof of Theorem~\ref{tm:rigidLattes}]
If $\dim_{\mathcal{M}}(f_\lambda,\Lambda)=0$, then by Theorem~\ref{tm:rigidlattesnewversion}, $0=q_{\sqrt d}+q_{-\sqrt d}$.  Since both $q_{\pm\sqrt d}\geq0$, we must have $q_{\sqrt d}=q_{-\sqrt d}=0$.

On the other hand, suppose $\dim_{\mathcal{M}}(f_\lambda,\Lambda)>0$.  Again by Theorem~\ref{tm:rigidlattesnewversion},  at least one of $q_{\sqrt d}$ and $q_{-\sqrt d}$ must be positive.  Recall that $U$, the linear part of $A_t$, is unitary.  Thus, we have $q_m:=\dim_\C E_{m}(\sqrt{d}\cdot U)=0$ for all $m\neq\pm\sqrt{d}$ and $q_m=0$ for all $m$ if $d$ is not the square on an integer.  Since one or both of $q_{\pm\sqrt d}$ must be positive, the result follows.
\end{proof}

\section{Symmetric product Latt\`es maps}\label{sec:symLattes}

In this section, we prove Theorem \ref{tm:symmlattes}.  First, though, note that among symmetric products, we can easily characterize Latt\`es ones.

\begin{proposition}
Let $k,d\geq2$ and let $f\in\Hom_d(\p^1)$, and let $F\in\Hom_d(\p^{k})$ be the $k$-fold symmetric product of $f$. Then $F$ is a Latt\`es map if and only if $f$ is a Latt\`es map. Moreover, if $\mathcal{T}$ is an elliptic curve, $\mathcal{I}:\mathcal{T}\to\mathcal{T}$ is affine and $\Theta:\mathcal{T}\to\p^1$ is the Galois cover such that
\begin{center}
$\xymatrix {\relax
\mathcal{T} \ar[r]^{\mathcal{I}} \ar[d]_{\Theta} & \mathcal{T}\ar[d]^{\Theta} \\
\p^1 \ar[r]_{f} & \p^1}$
\end{center}
is a commutative diagram, then the following diagram commutes:
\begin{center}
$\xymatrix {\relax
\mathcal{T}^k \ar[r]^{(\mathcal{I},\ldots,\mathcal{I})} \ar[d]_{\eta_k\circ(\Theta,\ldots,\Theta)} & \mathcal{T}^k\ar[d]^{\eta_k\circ(\Theta,\ldots,\Theta)} \\
\p^{k} \ar[r]_{F} & \p^{k}}$
\end{center}
\label{prop:lattes}
\end{proposition}

\begin{proof}
The characterization of $L(F)$, the sum of Lyapunov exponents of $F$, in terms of $L(f)$ from Lemma~\ref{lm:lyap} provides a path for a simple proof.  Berteloot and Dupont \cite{bertelootdupont} proved that $F$ is Latt\`es if and only $L(F)=k\log \sqrt{d}$. A similar result was known by \cite{zdunik} in dimension $1$. Namely, $f\in\Hom_d(\p^1)$ is Latt\`es if and only if $L(f)=\log \sqrt{d}$. The fact that $F$ is Latt\`es if and only if $f$ is Latt\`es follows directly from Lemma~\ref{lm:lyap}. Assume now that $f$ is a Latt\`es map induced by $\mathcal{I}:\mathcal{T}\to\mathcal{T}$ under the Galois cover $\Theta:\mathcal{T}\to\p^1$. By construction of $F$, the following diagram commutes:
$$\xymatrix {\relax
\mathcal{T}^k \ar[r]^{(\mathcal{I},\ldots,\mathcal{I})} \ar[d]_{(\Theta,\ldots,\Theta)} & \mathcal{T}^k\ar[d]^{(\Theta,\ldots,\Theta)} \\
(\p^1)^{k} \ar[d]_{\eta_k}\ar[r]_{(f,\ldots,f)} & (\p^1)^{k}\ar[d]^{\eta_k}\\
\p^{k} \ar[r]_{F} & \p^{k}}$$
which concludes the proof.
\end{proof}

We now come to the proof of Theorem~\ref{tm:symmlattes}.

\begin{proof}[Proof of Theorem~\ref{tm:symmlattes}]
It is classical that $f$ is rigid if and only if the linear part of the induced isogeny $A$ is not an integer (see~\cite{Milnor}). We now let $F$ be the $k$-fold symmetric product of $f$. We have the following alternatives:
\begin{enumerate}
\item either $q_{\sqrt{d}}+q_{-\sqrt{d}}=0$,
\item or $q_{\sqrt{d}}=k$ (or $q_{-\sqrt{d}}=k$).
\end{enumerate}
We now just have to apply Theorem~\ref{tm:rigidLattes}.
\end{proof}

\paragraph*{\textit{In dimension 2.}}
An \emph{algebraic web} is given by a reduced curve $C\subset(\p^2)^\star$, where $(\p^2)^\star$ is the dual projective plane consisting of lines in $\p^2$. The web is invariant for a holomorphic map $f$ on $\p^2$ if every line in $\p^2$ belonging to $C$ is mapped to another such line \cite{dabija2010algebraic}.

In~\cite[Theorems~4.2 $\&$ 4.4]{Rong}, Rong Feng gave the following description of Latt\`es maps on $\p^2$: pick a Latt\`es map $F\in\Hom_d(\p^2)$, then
\begin{itemize}
\item either $F$ or $F^2$ is the $2$-fold symmetric product of a Latt\`es map $f\in\Hom_d(\p^1)$ (resp. $f\in\Hom_{d^2}(\p^1)$),
\item or $F$, $F^2$, $F^3$ or $F^6$ is a holomorphic map preserving an algebraic web associated to a smooth cubic.
\end{itemize}

Combined with Theorem~\ref{tm:symmlattes}, this description directly gives the following result.

\begin{corollary}\label{cor:rigid}
Let $F\in\Hom_d(\p^2)$ be a $2$-fold symmetric product Latt\`es map. Then
\begin{enumerate}
\item either $F$ is a rigid $2$-fold symmetric product,
\item or $F$ is a $2$-fold symmetric product which belongs to a family of Latt\`es maps which has dimension $3$ in moduli. In that case, $F$ can be approximated by Latt\`es maps admitting a suitable iterate which preserve an algebraic web associated to a smooth cubic. More precisely, any maximal subfamily of symmetric products has dimension $1$ in moduli.
\end{enumerate}
\end{corollary}

\begin{proof}
According to Theorem~\ref{tm:symmlattes}, either $F$ is rigid or $F$ belongs to a family of Latt\`es maps which has dimension $3$ in moduli. Now, if $F$ belongs to such a family, there is a maximal subfamily $(F_t)_t$ which consists of symmetric products which has dimension $1$. The family $(F^m_t)_t$ is also a maximal family of symmetric products for all $m\geq1$. The result follows from Rong's classification.
\end{proof}

\paragraph*{\textit{Explicit examples.}}
We now give exlicit examples.

\begin{exmp}[Rigid Latt\`es maps of $\p^2$ and $\p^3$]
In \cite{Milnor}, Milnor provided an example of degree 2 Latt\`es map $f\in\rat_2$:
\[f\colon[z:t]\mapsto[z^2+a^2zt:t^2+a^2zt]\]
(where $a=i\sqrt2$, $a=(1\pm\sqrt7)/2$, or $a=1\pm i$).  As $\deg(f)$ is not the square of an integer, $f$ has to be isolated in $\mathcal{M}_2$ (see e.g. \cite{Milnor}). The 2-fold symmetric product of $f$ is $F\colon\mathbb P^2\rightarrow\mathbb P^2$ given by
\[\left[z:t:w\right] \mapsto \left[F_1:F_2:F_3\right]~,\]
where $F_1(z,t,w)=z^2+a^2zt+a^4zw$, $F_2(z,t,w)=t^2+a^2zt+a^2tw+2(a^4-1)zw$ and $F_3(z,t,w)=w^2+a^2tw+a^4zw$. We now turn to the $3$-fold symmetric product of $f$.

The 3-fold symmetric product of $f$ is $G\colon\mathbb P^3\rightarrow\mathbb P^3$ given by
\[\left[z:t:w:u\right] \mapsto \left[G_1:G_2:G_3:G_4\right],\]
where
\begin{eqnarray*}
G_1(z,t,w,u)&=&z^2+a^2zt+a^4zw+a^6zu,\\
G_2(z,t,w,u)&=&t^2 + a^4tu + a^2tw + a^2zt + 3a^6zu + 2a^4zw - 3a^2zu - 2zw,\\
G_3(z,t,w,u)&=&w^2 + a^4zw + a^2tw + a^2wu + 3a^6zu  + 2a^4tu  - 3a^2zu  - 2tu,\\
G_4(z,t,w,u)&=& u^2 + a^2wu + a^4tu + a^6zu.
\end{eqnarray*}
By Theorem \ref{tm:rigidLattes}, since $f$ is rigid, both maps $F$ and $G$ are rigid Latt\`es maps, respectively, in the moduli spaces $\mathcal{M}_2(\p^2)$ and $\mathcal{M}_2(\p^3)$.
\end{exmp}

\begin{exmp}[Flexible Latt\`es maps of $\p^2$]
The one-dimensional family $(f_\lambda)_{\lambda\in\C\setminus\{0,1\}}$ of degree $4$ maps
\[f_\lambda:[z:t]\in\p^1 \, \longmapsto \, [(z^2-\lambda t^2)^2:4zt(z-t)(z-\lambda t)]\]
is a family of flexible Latt\`es maps of $\p^1$. The family of $2$-fold symmetric products $F_\lambda:\p^2\rightarrow\p^2$ given for $\lambda\in\C\setminus\{0,1\}$ by
\[\left[z:t:w\right] \mapsto \left[F_{1,\lambda}:F_{2,\lambda}:F_{3,\lambda}\right],\]
where
\begin{eqnarray*}
F_{1,\lambda}(z,t,w)&=& ((z+\lambda w)^2-\lambda t^2)^2,\\
F_{2,\lambda}(z,t,w)&=& 4((z+\lambda w)^3t+2(\lambda+1)(z+\lambda w)^2zw+\lambda(z+\lambda w)t^3-8\lambda zwt(z+\lambda w)\\
& & -(\lambda+1)(z^2t^2+\lambda^2t^2w^2))~,\\
F_{3,\lambda}(z,t,w)&=& 16zw(z-t+w)(z-\lambda t+\lambda^2w)
\end{eqnarray*}
is a family of flexible Latt\`es maps of $\p^2$. Moreover, by Corollary \ref{cor:rigid}, it is a \emph{strict subfamily} of a $2$-dimensional family of flexible Latt\`es maps.
\end{exmp}


\begin{thebibliography}{KPR}

\bibitem[A]{astorg}
Matthieu Astorg.
\newblock Dynamics of post-critically finite maps in higher dimension.
\newblock {\em Ergodic Theory and Dynamical Systems}, pages 1--20, 2018.

\bibitem[BB]{BB1}
Giovanni Bassanelli and Fran{\c{c}}ois Berteloot.
\newblock Bifurcation currents in holomorphic dynamics on {$\mathbb{P}^k$}.
\newblock {\em J. Reine Angew. Math.}, 608:201--235, 2007.

\bibitem[BBD]{BBD}
Fran\c{c}ois Berteloot, Fabrizio Bianchi, and Christophe Dupont.
\newblock Dynamical stability and lyapunov exponents for holomorphic
  endomorphisms of $\mathbb{CP}(k)$.
\newblock {\em Annales Scientifiques de l'ENS}, to appear.
\newblock preprint arXiv : math.DS/1403.7603.

\bibitem[BD1]{bertelootdupont}
Fran\c{c}ois Berteloot and Christophe Dupont.
\newblock Une caract\'erisation des endomorphismes de {L}att\`es par leur
  mesure de {G}reen.
\newblock {\em Comment. Math. Helv.}, 80(2):433--454, 2005.

\bibitem[BD2]{briendduval}
Jean-Yves Briend and Julien Duval.
\newblock Exposants de {L}iapounoff et distribution des points p\'eriodiques
  d'un endomorphisme de {$\mathbf{C}{\rm P}^k$}.
\newblock {\em Acta Math.}, 182(2):143--157, 1999.

\bibitem[BG]{Article2}
Xavier Buff and Thomas Gauthier.
\newblock Pertubations of flexible {L}att\`es maps.
\newblock {\em Bull. Soc Math. France}, 141(4):603--614, 2013.

\bibitem[BL1]{BL}
F.~Berteloot and J.-J. Loeb.
\newblock Spherical hypersurfaces and {L}att\`es rational maps.
\newblock {\em J. Math. Pures Appl. (9)}, 77(7):655--666, 1998.

\bibitem[BL2]{BertelootLoeb}
Fran{\c{c}}ois Berteloot and Jean-Jacques Loeb.
\newblock Une caract{\'e}risation g{\'e}om{\'e}trique des exemples de
  latt{\`e}s de {${\bf P}^k$}.
\newblock {\em Bulletin de la Soci{\'e}t{\'e} math{\'e}matique de France},
  129(2):175--188, 2001.

\bibitem[De]{Debarre}
Olivier Debarre.
\newblock {\em Tores et vari\'et\'es ab\'eliennes complexes}, volume~6 of {\em
  Cours Sp\'ecialis\'es [Specialized Courses]}.
\newblock Soci\'et\'e Math\'ematique de France, Paris; EDP Sciences, Les Ulis,
  1999.

\bibitem[Du]{Dupont-spherical}
Christophe Dupont.
\newblock Exemples de {L}att\`es et domaines faiblement sph\'eriques de
  {$\mathbb{C}^n$}.
\newblock {\em Manuscripta Math.}, 111(3):357--378, 2003.

\bibitem[DJ]{dabija2010algebraic}
Marius Dabija and Mattias Jonsson.
\newblock Algebraic webs invariant under endomorphisms.
\newblock {\em Publicacions Matem{\`a}tiques}, pages 137--148, 2010.

\bibitem[DS]{dinhsibony2}
Tien-Cuong Dinh and Nessim Sibony.
\newblock Dynamics in several complex variables: endomorphisms of projective
  spaces and polynomial-like mappings.
\newblock In {\em Holomorphic dynamical systems}, volume 1998 of {\em Lecture
  Notes in Math.}, pages 165--294. Springer, Berlin, 2010.

\bibitem[FS]{fornaesssibony}
John~Erik Forn{\ae}ss and Nessim Sibony.
\newblock Dynamics of {${\bf P}^2$} (examples).
\newblock In {\em Laminations and foliations in dynamics, geometry and topology
  ({S}tony {B}rook, {NY}, 1998)}, volume 269 of {\em Contemp. Math.}, pages
  47--85. Amer. Math. Soc., Providence, RI, 2001.

\bibitem[J]{Jonsson-finite}
Mattias Jonsson.
\newblock Some properties of {$2$}-critically finite holomorphic maps of
  {$\mathbf{P}^2$}.
\newblock {\em Ergodic Theory Dynam. Systems}, 18(1):171--187, 1998.

\bibitem[KPR]{kaschner}
Scott~R Kaschner, Rodrigo~A P{\'e}rez, and Roland~KW Roeder.
\newblock Examples of rational maps of $\mathbb{P}^2$ with equal dynamical
  degrees and no invariant foliation.
\newblock {\em Bulletin de la SMF}, 144(2):279--297, 2016.

\bibitem[L]{levyPn}
Alon Levy.
\newblock The space of morphisms on projective space.
\newblock {\em Acta Arith.}, 146(1):13--31, 2011.

\bibitem[Ma]{maakestad}
Helge Maakestad.
\newblock Resultants and symmetric products.
\newblock {\em Comm. Algebra}, 33(11):4105--4114, 2005.

\bibitem[Mi]{Milnor}
John Milnor.
\newblock On {L}att\`es maps.
\newblock In {\em Dynamics on the {R}iemann sphere}, pages 9--43. Eur. Math.
  Soc., Z\"urich, 2006.

\bibitem[R]{Rong}
Feng Rong.
\newblock Lattes maps on $\mathbb{P}^2$.
\newblock {\em Journal de math{\'e}matiques pures et appliqu{\'e}es},
  93(6):636--650, 2010.

\bibitem[S]{Sibony}
Nessim Sibony.
\newblock Dynamique des applications rationnelles de {$\textbf{P}^k$}.
\newblock In {\em Dynamique et g\'eom\'etrie complexes ({L}yon, 1997)},
  volume~8 of {\em Panor. Synth\`eses}, pages ix--x, xi--xii, 97--185. Soc.
  Math. France, Paris, 1999.

\bibitem[U1]{ueda2}
Tetsuo Ueda.
\newblock Complex dynamical systems on projective spaces.
\newblock {\em Chaotic Dynamical Systems}, 13:120--138, 1993.

\bibitem[U2]{ueda}
Tetsuo Ueda.
\newblock Critical orbits of holomorphic maps on projective spaces.
\newblock {\em J. Geom. Anal.}, 8(2):319--334, 1998.

\bibitem[Z]{zdunik}
Anna Zdunik.
\newblock Parabolic orbifolds and the dimension of the maximal measure for
  rational maps.
\newblock {\em Invent. Math.}, 99(3):627--649, 1990.

\end{thebibliography}
\end{document}